\documentclass[12pt]{article}
\usepackage{amsfonts}
\usepackage{amsmath}

\usepackage{amssymb}

\usepackage[a4paper,top=2cm,bottom=2cm,left=2cm,right=2cm]{geometry}
\newtheorem{theorem}{Theorem}[section]

\newtheorem{corollary}[theorem]{Corollary}

\newtheorem{lemma}[theorem]{Lemma}

\newtheorem{proposition}[theorem]{Proposition}

\newenvironment{proof}[1][Proof]{\noindent\textbf{#1.} }{\ \rule{0.5em}{0.5em}}

\begin{document}
\title{A real chain condition for groups}
\author{Ulderico Dardano\ \ {\small \textsl{and}}\ \ Fausto De Mari}
\date{}
\maketitle

\begin{center}
\vskip -0.95cm \texttt{{\small \phantom{ ......}dardano@unina.it%
\phantom{
.....}\ fausto.demari@unina.it}}\\[0pt]
\vskip 0.3cm Universit\`a degli Studi di Napoli ``Federico II'' \\[0pt]
\end{center}

\abstract{\noindent  We consider a very weak chain condition for a
poset, that is the absence of subsets which are order isomorphic to the set
of real numbers in their natural ordering; we study generalised radical groups in which  this
finiteness condition is set on the poset of subgroups which do not have certain 
properties which are generalizations of normality. This completes many
previous results which considered (apparently) stronger chain conditions.}



\section{Introduction}

The study of groups $G$ with some finiteness condition on the poset of subgroups of $G$ with (or without) given  properties has been object of many investigations; see for example \cite{DS} for a survey up
to 2009 and \cite{DDM23,DDMR,DM07,DM18,DM20,GKR,GKR1,KS,T03} for more recent contributions.
The most elegant and popular ones have been perhaps those concerning the
restrictions on the \emph{chains} - i.e. totally ordered subsets - of $\chi$ (or non-$\chi$) subgroups. Here, as throughout the paper, the letter $\chi$ denotes a
property pertaining to subgroups.

\medskip
Recall that a poset $P$ of subgroups a group is said to have the {\it weak maximal condition} (Max-$\infty$), the {\it weak minimal condition} (Min-$\infty$), the {\it weak double chain condition} (DCC-$\infty$), if
$P$ does not contain chains whose order type is the same of the
natural numbers $\mathbb{N}$, as the negative integers $-\mathbb{N}$, as
the whole set of integers $\mathbb{Z}$ respectively and, furthermore, in which each  subgroup  has  infinite index in the successive (see \cite{
Z}).  When $P$ is the poset of all $\chi$ (or non-$\chi$) subgroups of a group $G$, it is usually said that $G$ has the weak maximal condition on $\chi$ (non-$\chi$, respectively) subgroups whenever $P$ has Max-$\infty$. The corresponding terminology is used for Min-$\infty$ and DCC-$\infty$.

\medskip
 It is clear that Max-$\infty$, as well as Min-$\infty$, imply DCC-$\infty$, and many investigations have been carried out in order to find for which poset $P$ of subgroups (or, equivalentely, for which property $\chi$) these conditions are all equivalent. However, due to the existence of the so-called Tarski groups (i.e. infinite simple groups whose proper non-trivial subgroups have prime order), investigations in this area  are usually completed under a suitable (generalised) solubility condition. Zaicev \cite{Z} proved that a locally (soluble-by-finite) group has DCC-$\infty$  (on all subgroups) if and only if it has Max-$\infty$ or Min-$\infty$; moreover, any locally (soluble-by-finite) group with DCC-$\infty$ is a soluble-by-finite minimax group. Recall that a group is said to be {\it minimax} if it has a finite series whose factor satisfy either the minimal or the maximal condition.

\medskip
We say that a poset of subgroups of a group has the {\it real chain condition} ($RCC$) if it does not contain chains whose order type is the same as the set of real numbers $\mathbb{R}$  with their usual ordering. Note that DCC-$\infty$ implies RCC (see Proposition \ref{WDDC_implies_deviation}, below). Thus, we are interested in finding for which properties $\chi$ the following holds. 

\medskip
\noindent \textbf{Framework Statement$\;\;\;$}{\it Let  $\chi$ be a
property pertaining to subgroups. For a generalised radical group, the following are equivalent:
\begin{itemize}
\item[(i)]  the weak minimal condition on non-$\chi$ subgroups; 

\item[(ii)] the weak maximal condition on non-$\chi$ subgroups; 

\item[(iii)] the weak double chain condition on non-$\chi$ subgroups.

\item[(iv)]  the real chain condition on non-$\chi$ subgroups;
\end{itemize}}

\medskip
Recall that a group is said {\it generalised radical} when it has an ascending (normal) series with locally (nilpotent or finite) factors.

\medskip
In the following sections, it will be shown that the above Framework Statement holds when $\chi$ is one the properties: $n$=normal; $an$=almost normal; $nn$=nearly normal; $m$=modular; $per$=permutable. Recall that for each of the previous choiches for the property $\chi$ and for a generalised radical group $G$,  the equivalence between the first three condition in the above Framework Statement is already known and it is also known that, with the exception of $\chi=an$ or $\chi=nn$, each of these three items is equivalent to the property that {\em either $G$ is a soluble-by-finite minimax group or all subgroups of $G$ have $\chi$}  (see \cite{CK,DM07,DM18,KG}). We will also prove the Framework Statement holds when $\chi$ is the property for a subgroup to be subnormal (but only in periodic soluble case). We will give futher references, detailed statements
and proofs for the different properties in the following corresponding
sections. For undefined terminology and basic results we refer to \cite
{R72,R96}.

\medskip
Until now, the most general chain condition in the theory of infinite groups appears to be the condition that the
poset of non-$\chi$ subgroups has the (set
theoretical) {\it deviation} (see \cite{GKR,GKR1,KS,T03}). We do not recall here
the (recursive) definition of the deviation since, for our purposes, it is
enough to recall the fact that a poset has deviation if and only if it
contains no sub-poset order isomorphic to the poset $D$ of all dyadic rationals $m/2^n$ in the interval from $0$ to $1$ (see \cite{CR}, 6.1.3). Since $D$ is a countable dense poset without endpoints, it is order-isomorphic to the rational numbers by Cantor's isomorphism theorem. 
Therefore  \emph{a poset has deviation if and only if it
contains no sub-poset order isomorphic to the poset $\mathbb{Q}$ of rational
numbers} in their usual ordering. Clearly, any poset with deviation has RCC and the converse holds when the poset is complete but it does not seem always to be the case; moreover, any poset  of subgroups with DCC-$\infty$ has deviation and so also RCC (see Proposition \ref{WDDC_implies_deviation}, below). Therefore as a consequence of our Framework Statement, the equivalence between RCC and deviation is proved true for the poset of all non-$\chi$ subgroups of a group for our selection of properties\;$\chi$.

\section{Preliminary results}

For brevity, we call $\mathbb{Z}$-, $\mathbb{Q}$- or $\mathbb{R}$-chain a poset with
the same order type as $\mathbb{Z}$, $\mathbb{Q}$ or $\mathbb{R}$ respectively.

\medskip
\begin{proposition}
\label{WDDC_implies_deviation} \label{compl}Let $P$ be a poset of subgroups of a group $G$. Then
\begin{itemize}
\item[(i)] If $P$ has DCC-$\infty$, then $P$ has deviation.
\item[(ii)] If $P$ has deviation, then $P$ has RCC.
\item[(iii)] If $P$ is complete and has RCC, then $P$ has deviation.
\end{itemize}
\end{proposition}

\begin{proof}
$(i)$ Assume that $P$ has not deviation, so it contains a strictly increasing family 
$(X_i)_{i\in\mathbb{Q}}$ of subgroups, then $(X_i)_{i\in\mathbb{Z}}$ is a $%
\mathbb{Z}$-chain in which for each $i\in\mathbb{Z}$ the index $%
|X_{i+1}:X_i| $ is infinite because there are infinitely many subgroups $X_j$
with $i<j<i+1 $ and $j\in\mathbb{Q}$. Hence $P$ has not DCC-$\infty$.

$(ii)$ is trivial.

$(iii)$ Assume that $P$ does not contain $\mathbb{R}$-chains and, for a
contradiction, let $(X_i)_{i\in\mathbb{Q}}$ be a $\mathbb{Q}$-chain of
elements of $P$. Since $P$ is complete, if $Y_r=\sup\{X_i:i\leq r\}$ for
every $r\in \mathbb{R}$, we obtain that $(Y_i)_{i\in\mathbb{R}}$ is an $%
\mathbb{R}$-chain of $P$, a contradiction.
\end{proof}

\medskip

Since the poset of \textsl{all }subgroups of a group is complete, it follows by
Proposition \ref{WDDC_implies_deviation} that a group has RCC for 
\textsl{all} subgroups if and only if it has deviation. On the other hand, any
direct product of infinite non-identity groups does contain an $\mathbb{R}$%
-chain as we state in next elementary lemma.

\medskip

\begin{lemma}
\label{lemma0} If a group $G$ is the direct product of infinitely many
non-trivial subgroups, then $G$ has not RCC. 
\end{lemma}

\begin{proof}
Clearly $G$ contains the direct product of countably many non-trivial
subgroups and hence a subgroup of the form $\underset{i\in\mathbb{Q}}{%
\mathrm{Dr}}{\ G_{i}}$. Then, an $\mathbb{R}$-chain is formed by the
subgroups $G_r=\underset{i<r}{\mathrm{Dr}}{\ G_{i}}$, for $r\in\mathbb{R}$.
\end{proof}

\bigskip
Tushev proved that \emph{a soluble group
has deviation if and only it is minimax} (see \cite{T03}, Lemma 4.4). This result can be extended as in
Theorem \ref{TheoremA} below. Before, let us state a standard general result
that will be used in what follows to reduce our investigation to
radical-by-finite groups, where \textit{radical} means that the group has an ascending (normal) series whose factors are locally nilpotent.

\medskip
\begin{proposition}
\label{GenRad}Let $\mathfrak{X}$ be a class of groups which closed with
respect to forming subgroups and homomorphic images, such that each locally
finite $\mathfrak{X}$-group is soluble-by-finite. Then any generalised
radical $\mathfrak{X}$-group $G$ is radical-by-finite.
\end{proposition}

\begin{proof}
Let $K$ a subgroup of $G$ which is maximal with respect to be radical and
normal in $G$, and let $T/K$ be a subgroup of $G/K$ which is maximal with
respect to be locally finite and normal in $G/K$. By the properties of $%
\mathfrak{X}$ we have that there is a normal subgroup $S$ of $T$ containing $
K$ such that $S/K$ is soluble and $T/S$ is finite. Then $S^G/K$ is soluble
and so, by the maximality of $K$, we have that $S^G=K$. Hence $T/K$ is finite
and $T$ is radical-by-finite.

Let $H/T$ be any normal subgroup of $G/T$ which is locally nilpotent, then $%
C_{H/K}(T/K)$ is a normal locally nilpotent subgroup of $G/K$. Since $G/K$
has no non-trivial locally nilpotent normal subgroups, we have that $%
C_{H/K}(T/K)$ is trivial; hence $H/K$ is finite. It follows that $H=T$.
Therefore $G/T$ has no non-trivial subgroups which are either locally
nilpotent or locally finite; on the other hand, $G$ is generalised radical
and hence $G=T$ is radical-by-finite.
\end{proof}

\bigskip
\begin{theorem}
\label{TheoremA} Let $G$ be a generalised radical group. Then the poset of
all subgroups of $G$ has RCC if and only if $G$ is a soluble-by-finite minimax group.
\end{theorem}

\begin{proof}
If $G$ is minimax, it is extension of groups with either the minimal or the
maximal condition and hence $G$ certainly has RCC, as the property RCC is closed under extensions as a standard argument shows.

Conversely, notice first that in any group with RCC, each abelian subgroup
is minimax by Proposition \ref{WDDC_implies_deviation} and the already quoted result by Tushev (see \cite{T03}, Lemma 4.4). In particular, any locally finite group with RCC is a
Chernikov group by a celebrated result by Shunkov \cite{Shunkov} (and hence
is soluble-by-finite). Therefore if $G$ is a generalised radical group with
RCC, then $G$ is radical-by-finite by Proposition \ref{GenRad} and so it is a soluble-by-finite minimax group (see \cite{R72} Part 2, Theorem\;10.35).
\end{proof}

\medskip
Let us state now a technical general key lemma which will be useful later.

\medskip
\begin{lemma}
\label{1} Let $G$ be a group having a section $H/K$ which is the direct product of an infinite collection $(H_{\lambda }/K)_{\lambda \in \Lambda }$ of non-trivial subgroups, and let $L$ be a
subgroup of $G$ such that $L\cap H\leq K$ and $\left\langle H_{\lambda
},L\right\rangle =H_{\lambda }L$ for each $\lambda $.
If there is no $\mathbb{R}$-chain of non-$\chi$ subgroups of $G$ in the
interval $[H/K]$, then there exists a normal subgroup $H^{\ast }$ of $H$
containing $K$ such that $LH^{\ast }=H^{\ast }L$ is a $\chi $-subgroup of $G$.
\end{lemma}

\begin{proof}
Clearly the set $\Lambda$ may be assumed to be countable, so that it can be
replaced by the set $\mathbb{Q}$ of the rationals. Consider the subgroup $%
K_r=\underset{i<r}{\mathrm{Dr}}H_{i}$ for each $r\in\mathbb{R}$; then $%
\left\langle K_{r},L\right\rangle =K_{r}L$ for each $r\in\mathbb{R}$. Let $%
r_1,r_2\in\mathbb{R}$ with $r_1<r_2$, then $K_{r_1}< K_{r_2}$. If were $%
K_{r_1}L= K_{r_2}L$, since $L\cap K_{r_2}\leq L\cap H\leq K\leq K_{r_1}$,
Dedekind's Modular Law would give that 
\begin{equation*}
K_{r_2}=K_{r_2}L\cap K_{r_2}=K_{r_1}L\cap K_{r_2}=K_{r_1}(L\cap
K_{r_2})=K_{r_1};
\end{equation*}
this contradiction proves that $K_{r_1}L< K_{r_2}L$. Therefore $K_rL$ must
be a $\chi $-subgroup of $G$ for some $r\in\mathbb{R}$ and the lemma holds
with $H^{\ast }=K_r$.
\end{proof}

\medskip 
The next result applies when $\chi$ is the property of being a normal
subgroup or, more generally, when $\chi$ is the property of being $\Gamma$
-invariant for some subgroup $\Gamma $ of the automorphism group of the
group. It also holds when $\chi $ is one of the the properties $an$, $nn$
(see \cite{M}, Lemma 1), $sn$ (see \cite{R96}, 13.14 and 13.1.5);
moreover, item $(i)$ holds also for the properties $m$ and $per$ (see \cite{S94}, pp.201-202).

\medskip

\begin{lemma}
\label{2} Let $G$ be a group with RCC on non-$\chi$ subgroups. If $G$ contains a section $H/K$ which is the
direct product of an infinite collection of non-trivial subgroups, then the
following hold:

\begin{itemize}
\item[(i)] if $\chi $ is such that $\left\langle X,Y\right\rangle$ is a $%
\chi $-subgroup of $G$ whenever $X$ and $Y$ are $\chi $-subgroups of $G$ such that $X^Y=X$, 
then $H$ is a $\chi$-subgroup of $G$;

\item[(ii)] if $\chi $ is such that the intersection $X\cap Y$ is a $\chi $%
-subgroup of $G$ whenever $X$ and $Y$ are $\chi $-subgroups of $G$, then $K$
is a $\chi $-subgroup of $G$.
\end{itemize}
\end{lemma}

\begin{proof}
Write $H/K=H_{1}/K\times H_{2}/K$ where both $H_{1}/K$ and $H_{2}/K$ are the
direct product of an infinite collection of non-trivial subgroups.
Application of Lemma \ref{1} yields that there exist an $H_1$-invariant
subgroup $H_{1}^{\ast }$ in $[H_{1}/K]$ and an $H_2$-invariant subgroup $%
H_{2}^{\ast }$ in $[H_{2}/K]$ such that both $H_{1}^{\ast }H_{2}$ and $%
H_{1}H_{2}^{\ast }$ are $\chi $-subgroups of $G$. Again Lemma \ref{1} (with $%
L=\{1\}$) gives that here exist a subgroup $K_{1}^{\ast }$ in $[H_{1}/K]$
and a subgroup $K_{2}^{\ast }$ in $[H_{2}/K]$ such that both $K_{1}^{\ast }$
and $K_{2}^{\ast }$ are $\chi $-subgroups of $G$. Clearly $H_{1}^{\ast
}H_{2} $ and $H_{1}H_{2}^{\ast }$ are both normal subgroups of $H$, so that $%
(H_{1}H_{2}^{\ast })^{(H_{1}^{\ast }H_{2})}=H_{1}H_{2}^{\ast }$. Since $%
H=\left\langle H_{1}H_{2}^{\ast },H_{1}^{\ast }H_{2}\right\rangle $ and $%
K_{1}^{\ast }\cap K_{2}^{\ast }=K$, the lemma is proved.
\end{proof}

\medskip

\begin{lemma}
\label{catena-} Let $G$ be a group with RCC on non-$\chi$ subgroups, where $\chi$ is such that the intersection $X\cap Y$ is a $%
\chi $-subgroup whenever $X$ and $Y$ are $\chi $-subgroups. 
Let  $L$ be any
subgroup of $G$. If there exists a subgroup $H$ of $G$ which is the
direct product of an infinite collection of $L$-invariant  non-trivial 
subgroups and such that $L\cap H=\{1\}$, then $L$ is a $\chi $-subgroup of $G$.
\end{lemma}

\begin{proof}
Write $H=H_{1}\times H_{2}$ where both $H_{1}$ and $H_{2}$ are the direct
product of an infinite collection of non-trivial subgroups. Application of
Lemma \ref{1} yields that there exist subgroups $H_{1}^{\ast }\leq H_{1}$
and $H_{2}^{\ast }\leq H_{2}$ such that both $H_{1}^{\ast }L$ and $%
H_{2}^{\ast }L$ are $\chi $-subgroups of $G$. Therefore $L=H_{1}^{\ast
}L\cap H_{2}^{\ast }L$ is likewise a $\chi $-subgroup of $G$.
\end{proof}

\medskip
Finally, we state as a lemma a property of abelian groups which is probabily well-known and that we will use in our aurgument without further mention. For a proof of such a property see, for instance, Lemma 3.2 of \cite{GKR}.

\begin{lemma}Any abelian group which is not minimax has an homomorphic image which is the direct product of infinitely many non-trivial subgroups.
\end{lemma}

\medskip
\section{Real chain condition on non-normal subgroups}

Let $G$ be a group. The  {\it $FC$-centre} of $G$ is the subgroup consisting of all elements having finitely many conjugates, and $G$ is said to be an {\it $FC$-group} if it coincides with its $FC$-centre.  The class of all $FC$-groups have been widely studied. It tourns out, in particular, that if $G$ is an $FC$-group then $G/Z(G)$ and $G'$ are locally finite (see \cite{T}, Theorem 1.4 and Theorem 1.6), moreover $G/Z(G)$ is residually finite (see \cite{T}, Theorem 1.9).

A subgroup $H$ of $G$ is called {\it nearly normal} if the index $|H^G:H|$ is finite. Any group whose (cyclic) subgroups are nearly normal is an $FC$-group (see \cite{T}, Lemma 7.12), moreover {\it all (abelian) subgroups of a group are nearly normal if and only if the group is finite-by-abelian} (see \cite{T}, Theorem 7.17).

A subgroup $H$ of $G$ is called {\it almost normal} if it has finitely many conjugates in $G$, i.e. when the index $|G:N_G(H)|$ is finite. Clearly, if all (cyclic) subgroups of $G$ are almost normal then $G$ is an $FC$-group; moreover, {\it all (abelian) subgroups of a group are almost normal if and only if the group is central-by-finite} (see \cite{T}, Theorem 7.20). Notice that any central-by-finite group is finite-by-abelian (see \cite{T}, Theorem 1.2).

\medskip

\begin{lemma}
\label{nnFC}Let $G$ be a group with RCC on non-(almost normal) (resp. non-(nearly normal) subgroups. If $G$ is an $FC$-group, then $G/Z(G)$ is finite (resp. $G^{\prime }$ is finite).
\end{lemma}

\begin{proof}
The factor $G/Z(G)$ is periodic and so we may consider a torsion-free subgroup $F$ of $Z(G)$ such that $G/F$ is periodic. Since $G'$ is periodic, $%
Z(G/F)=Z(G)/F$ and so, replacying $G$ by $G/F$, it can be supposed that $G$
is periodic. Let $A$ be any abelian subgroup of $G$. Assume first that $A$
is a Chernikov group. Since $G/Z(G)$ is residually finite, the finite
residual of $A$ is contained in $Z(G)$ and so $A/A_{G}$ is finite. Then $A^{G}/A_{G}$ is likewise finite (see \cite{T}, Lemma 1.3) and hence $A$ is
both nearly normal and almost normal in this case. Suppose now that $A$ is
not a Chernikov group, hence $A$ has an homomorphic image which is the
direct product of infinitely many non-trivial subgroups 
and hence $A$ is almost normal in $G$ (resp. nearly normal) by Lemma \ref{2}. Therefore all abelian subgroups of $G$ are almost normal (resp. nearly
normal) and so lemma follows by above quoted results.
\end{proof}

\bigskip Next three lemmas allows us to assume that abelian subgroups have  finite total rank; where the {\it total rank} of an abelian group is the sum of all $p$-ranks for $p=0$ or $p$ prime. Recall also that a well-know result of Kulikov states that {\it any subgroup of a direct product of cyclic subgroups is likewise a direct product of cyclic subgroups} (see \cite{F73}, Theorem 3.5.7), in what follows we make use of this result also without further reference.

\medskip

\begin{lemma}
\label{xxx} Let $G$ be a group with RCC on non-(almost normal) (resp. non-(nearly normal) subgroups,
and let $A$ be a subgroup which is the direct product of infinitely many
non-trivial cyclic subgroups. Then all subgroups of $A$ are almost normal
(resp. nearly normal) subgroups of $G$.
\end{lemma}

\begin{proof}
Let $X$ be any cyclic direct factor of $A$. Clearly we may write $A=X\times A_1$
where $A_1$ is not finitely generated, and so application of Lemma \ref%
{catena-} gives that $X$ is almost normal (resp. nearly normal) in $G$.
Therefore $A$ is contained in the $FC$-centre of $G$ and all finitely
generated subgroup of $A$ are almost normal (and nearly normal) in $G$. On
the other hand, if $Y$ is any subgroup of $A$ which is not finitely
generated, then $Y$ is likewise a direct product of cyclic subgroups, and hence $Y$ is almost normal (resp. nearly
normal) in $G$ by Lemma \ref{2}. Therefore all subgroups of $A$ are almost
normal (resp. nearly normal) in $G$.
\end{proof}

\medskip

\begin{lemma}
\label{anor} Let $G$ be a group and let $A$ be a normal subgroup of $G$ which is the direct product of infinitely many non-trivial cyclic subgroups. If all subgroups of $A$ are almost normal in $G$, then $A$ contains a
subgroup which is the direct produt of infinitely many finitely generated $G$-invariant non-trivial subgroups.
\end{lemma}

\begin{proof}
Let $A_1=\{1\}$ and assume that $G$-invariant subgroups $A_1,\dots,A_n$ of $A$ have been constucted in such a way that $\langle
A_1,\dots,A_n\rangle=A_1\times\cdots\times A_n$ is finitely generated. Then
there exists subgroups $X$ and $Y$ such that $Y$ is finitely generated, $%
\langle A_1,\dots,A_n\rangle\leq Y$ and $A=X\times Y$. %
Since $X$ has finitely many conjugates in $G$, the factor $A/X_G$ is
finitely generated; in particular, $X_G$ is not trivial and so we may choose
a non-trivial element $x\in X_G$. Since $A$ is contained in the $FC$-centre
of $G$, the subgroup $A_{n+1}=\langle x\rangle^G$ is finitely generated.
Therefore $\langle A_1,\dots,A_n, A_{n+1}\rangle=A_1\times\cdots\times
A_n\times A_{n+1}$ is a finitely generated subgroup of $A$, and so lemma
follows.
\end{proof}

\medskip

\begin{lemma}
\label{anrango}Let $G$ be a group with RCC on non-(almost normal) subgroups. If $G$ has a subgroup which is the direct product of infinitely many non-trivial cyclic
subgroups, then $G/Z(G)$ is finite.
\end{lemma}
\bigskip 
\begin{proof}We will prove firstly that $G$ contains a subgroup which is a direct product of infinitely many non-trivial normal subgroups.

Let $A$ be a subgroup of $G$ which is is the direct product of infinitely
many non-trivial cyclic subgroups. By Lemma \ref{catena-} follows easily that every cyclic
subgroup of $A$ is almost normal in $G$, hence $A$ is contained in the $FC$-centre $F$ of $G$. Since $F/Z(F)$ is finite by Lemma \ref{nnFC}, we may
clearly suppose that $A\leq Z(F)$. Let $T$ be the subgroup consisting of all
elements of finite order of $Z(F)$, and assume first that $T$ is not a
Chernikov group. Since $T$ is the direct product of its primary components,
which are normal subgroups of $G$, in order to prove our claim it can be
assumed that $\pi(T)$ is finite. Then there exists a prime $p$ such that the
Sylow $p$-subgroup $P$ of $T$ does not satisfy the minimal condition, so
that the socle of $P$ is an infinite abelian normal subgroup of $G$ of prime
exponent and hence application of Lemma \ref{xxx} and Lemma \ref{anor} give
us the required subgroup.
Assume now that $T$ is a Chernikov group, so that $Z(F)$ has infinite
torsion-free rank. Let $U$ be a free subgroup of $Z(F)$ such that $Z(F)/U$
is periodic; in particular, $U$ has infinite rank. Then $U$ is almost normal
in $G$ by Lemma \ref{2}, so that also $Z(F)/U_G$ is periodic. Thus $%
U_G\simeq U$ is a free abelian normal subgroup of infinite rank of $G$ and
again application of Lemma \ref{xxx} and Lemma \ref{anor} prove that $G$
contains the claimed subgroup.

Therefore $G$ contains a subgroups which is a direct product of infinitely many
non-trivial normal subgroups. Then it follows from Lemma \ref{catena-} that all cyclic subgroups are almost normal, so that $G$ is an $FC$-group and application of Lemma \ref{nnFC} concludes the proof.
\end{proof}

\bigskip

\begin{lemma}
\label{anfine}Let $G$ be a radical-by-finite group with RCC on non-(almost normal) subgroups. Then each non-minimax subgroup of $G$ is almost normal.
\end{lemma}

\begin{proof}
Let $H$ be any non-minimax subgroup of $G$, then $H$ contains an abelian
non-minimax subgroup $A$ (see \cite{R72} Part 2, Theorem 10.35). Let $B$ any
free subgroup of $A$ such that $A/B$ is periodic. If $B$ is not finitely
generated, then $G/Z(G)$ is finite by Lemma\;\ref{anrango} and so $H$ is
almost normal. Thus assume that $B$ is finitely generated, so that $A/B$
does not satisfy the minimal condition and hence its socle is infinite. Thus 
$B$ is almost normal by Lemma \ref{2}, so that also the periodic group $%
A/B_{G}$ has infinite socle and hence $G/B_{G}$ is finite over its
centre by Lemma \ref{anrango}. Since any central-by-finite group is also
finite-by-abelian, it follows that $G^{\prime }$ is polycyclic-by-finite.
Thus the abelian factor $H/H^{\prime }$ is not minimax and so it has an
homomorphic image which is the direct product of infinitely many non-trivial
subgroups; hence $H$ is almost normal in $G$ by Lemma \ref{2}.
\end{proof}

\bigskip

\begin{lemma}
\label{LFSFan} Let $G$ be a locally finite group with RCC on non-(almost normal) subgroups. Then
either $G$ is a Chernikov group or $G/Z(G)$ is finite. In particular, $G$ is
abelian-by-finite.
\end{lemma}

\begin{proof}
Assume that $G$ is not a Chernikov group. Then $G$ contains an abelian
subgroup $A$ which does not satisfy the minimal condition (see \cite{Shunkov}%
); thus the socle of $A$ is a direct product of infinitely many non-trivial
groups of prime order and hence $G/Z(G)$ is finite by Lemma \ref{anrango}.
\end{proof}

\medskip 
It has been proved in \cite{CK} that for a generalised radical group, weak minimal, weak maximal and weak double chain condition on non-(almost normal) subgroups are equivalent, moreover, a description of generalised radical groups statisfying such a condition is also given in the case of groups which are neither minimax nor central-by-finite.
%
%
%
Now we are in position to prove our Framework Statement when $\chi=an$, it add another equivalent condition to the weak chain conditions (and so also to the deviation) on non-(almost normal) subgroups. 

\bigskip

\begin{theorem}
\label{teoan} Let $G$ be a generalised radical group. Then the following are
equivalent:

\begin{itemize}

\item[(i)] $G$ satisfies the weak minimal condition on non-(almost normal)
subgroups;

\item[(ii)] $G$ satisfies the weak maximal condition on non-(almost normal)
subgroups;

\item[(iii)] $G$ satisfies the weak double condition on non-(almost normal)
subgroups.

\item[(iv)]  $G$ satisfies the real chain condition on non-(almost normal) subgroups.

\end{itemize}
\end{theorem}

\begin{proof} As already quoted, conditions $(i)$, $(ii)$ and $(iii)$ are equivalent, and imply $(iv)$ by Proposition \ref{WDDC_implies_deviation}. Conversely, if $G$ satisfies $(iv)$, then $G$ is radical-by-finite by Lemma \ref{LFSFan} and Proposition \ref{GenRad}, so that Lemma \ref{anfine} yields that each non-minimax subgroup of $G$ is
almost normal and hence Theorem 12 of \cite{CK} can be applied to conclude the
proof.
\end{proof}
%
%

%
%
%

%

\medskip
We turn to consider the case when $\chi=nn$. First step is to restrict the total rank of abelian subgroups.

\medskip
\begin{lemma}
\label{nor}\label{nnrango}Let $G$ be a group with RCC on non-(nearly normal) subgroups. If $G$ has a subgroup which is the direct product of infinitely many non-trivial cyclic subgroups, then $G^{\prime}$ is finite.
\end{lemma}

\begin{proof}
Let $A$ be a subgroup of $G$ which is is the direct product of infinitely
many non-trivial cyclic subgroups, then $A$ is a nearly normal subgroup of $G$ by Lemma~\ref{2}. Since it is well-know that any abelian-by-finite group has a
characteristic abelian subgroup of finite index, it follows that $A^G$
contains a $G$-invariant abelian subgroup $N$ of finite index. Clearly, $%
A\cap N$ has finite index also in $N$ so that $N$ is likewise a direct
product of infinitely many non-trivial cyclic subgroups (see \cite{F73},
Theorem 3.5.7 and Exercise 8 p.99). Replacing $A$ with $N$ it can be
supposed that $A$ is a normal subgroup of $G$. Moreover, all subgroups of $A$ are nearly normal subgroups of $G$ by Lemma \ref{xxx}.

Let $T$ be the subgroup consisting of all elements of finite order of $A$.
Then $T$ is normal in $G$ and $T$ is the direct product of non-trivial
cyclic subgroups by Kulikov's Theorem already quoted; 
moreover, all subgroups of $A/T$ are normal in $G/T$ (see \cite{Casolo}, Lemma~2.7). If $T$ is finite,
it follows easily from Lemma~\ref{catena-} that every cyclic subgroup of $%
G/T $ is nearly normal; hence $G/T$ is an $FC$-group and application of
Lemma \ref{nnFC} yields that $G^{\prime}$ is finite.

Therefore it can be assumed that $A=T$ is infinite. Then $A$ contains a $G$%
-invariant subgroup $D$ which is a finite-by-divisible such that all
subgroups of $A/D$ are normal in $G/D$ (see \cite{Casolo}, Theorem~2.11).
Since $A$ is the direct product of non-trivial cyclic subgroups, also $D$ is
likewise the direct product of non-trivial cyclic subgroups. 
Hence $D$ must be finite and so, as before, it can be
obtained that $G^{\prime}$ is finite.
\end{proof}

\medskip

\begin{lemma}
\label{LFSF} Let $G$ be a locally finite group with RCC on non-(nearly normal) subgroups. Then
either $G$ is a Chernikov group or $G^{\prime}$ is finite. In particular, $G$
is soluble-by-finite.
\end{lemma}

\begin{proof}
Assume that $G$ is not a Chernikov group. Then $G$ contains an abelian
subgroup $A$ which does not satisfy the minimal condition (see \cite{Shunkov}%
); thus the socle of $A$ is a direct product of infinitely many non-trivial
groups of prime order and hence $G^{\prime}$ is finite by Lemma \ref{nnrango}%
.
\end{proof}

\medskip 
In \cite{DM07}, it has been proved that for a generalised radical group, weak minimal, weak maximal and weak double chain condition on non-(nearly normal) subgroups are equivalent; moreover, with the exception of finite-by-abelian groups, it tourns out that for generalised radical groups, weak chain conditions on non-(nearly normal) subgroups are equivalent to weak chain conditions on non-(almost normal) subgroups. 
In next result we prove that Framework Statement holds when $\chi=nn$ so that, in particular, real chain condition is equivalent to the weak chain conditions (and so also to the deviation) for such a subgroup property.

\medskip

\begin{theorem}
\label{teonn} Let $G$ be a generalised radical group. Then the following are
equivalent:

\begin{itemize}

\item[(i)] $G$ satisfies the weak minimal condition on non-(nearly normal)
subgroups;

\item[(ii)] $G$ satisfies the weak maximal condition on non-(nearly normal)
subgroups;

\item[(iii)] $G$ satisfies the weak double condition on non-(nearly normal)
subgroups.

\item[(iv)]  $G$ satisfies the real chain condition on non-(nearly normal) subgroups.
\end{itemize}
\end{theorem}

\begin{proof} Since conditions $(i)$, $(ii)$ and $(iii)$ are equivalent (see \cite{DM07}, Theorem A) and imply $(iv)$ by Proposition \ref{WDDC_implies_deviation}, it is enough to prove that $(iv)$ implies $(iii)$. Let $G$ satisfy RCC on non-(nearly normal) subgroups. Lemma \ref{LFSF} and Proposition \ref{GenRad} give that $G$ is radical-by finite. By Lemma \ref{nnFC} it can be assumed that $G$ is not an $FC$-group so that $G$ does not contain subgroups which are a direct product of infinitely many non-trivial cyclic subgroups by Lemma \ref{nnrango}. Hence all abelian subgroups have finite total rank and so $G$ has a subgroup of finite index
having a finite series in which each factor is abelian of finite total rank
(see \cite{C}). It follows that $G$ has finite (Pr\"ufer) rank and so each
nearly normal subgroup is also almost normal (see \cite{dGR}, Lemma 3.1).
Therefore $G$ has RCC on non-(almost normal) subgroups, so that $G$ satisfy the weak double chain condition on non-(almost normal) subgroups by Theorem\;\ref{teoan} and thus also the weak double chain condition on non-(nearly normal) subgroups (see \cite{DM07}, Theorem 2.12).
\end{proof}

%
%
%

\medskip
Groups in which all subgroups are normal are well-known
since a long time and are the well described Dedekind groups (see \cite{R96}, 5.3.7). Moreover,
Kurdachenko and Goretski\u{\i} \cite{KG} showed that for locally
(soluble-by-finite) groups, the weak minimal condition on non-normal
subgroups is equivalent to the weak maximal condition on non-normal
subgroups, and any locally (soluble-by-finite) group satisfying such a
condition is either a soluble-by-finite minimax group or a Dedekind group (in particular, these result remains true for generalised radical groups by Proposition \ref{GenRad}). We extend this result to condition RCC and improve Corollary\;1 of \cite{GKR1} which handles only the periodic case.

\medskip

\begin{theorem}
\label{teon} Let $G$ be a generalised radical group. Then the following are
equivalent:

\begin{itemize}
\item[(i)] $G$ satisfies the weak minimal condition on non-normal
subgroups;

\item[(ii)] $G$ satisfies the weak maximal condition on non-normal
subgroups;

\item[(iii)] $G$ satisfies the weak double condition on non-normal
subgroups.

\item[(iv)]  $G$ satisfies the real chain condition on non-normal subgroups.

\item[(v)] either $G$ is a soluble-by-finite minimax group or all subgroups of $G$ are normal.
\end{itemize}
\end{theorem}

\begin{proof}
Conditions $(i)$, $(ii)$, $(iii)$ and $(v)$ are equivalent (see \cite
{KG}); moreover $(iii)$ implies $(iv)$ by Proposition \ref{WDDC_implies_deviation}. Hence assume that $G$ satisfy $(iv)$, and prove that $(i)$ holds. Notice that $G$ has RCC on non-(nearly normal) subgroups, so that $G$ is soluble-by-finite by Theorem \ref{teonn} and by Lemma 2.11 of \cite{DM07}.

Assume first that $G$ contains a subgroup $A$ which is the product of
countable many non-trivial cyclic subgroups; then $G^{\prime}$ is finite by
Lemma \ref{nnrango}. Let $X$ be any cyclic subgroup of $G$. Since $%
\left\vert G:C_{G}(X)\right\vert $ is finite, replacing $A$ with a suitable
subgroup, it can be assumed that $[A,X]=A\cap X=\{1\}$ and hence application
of Lemma \ref{catena-} to the subgroup $\langle X,A\rangle=A\times X$ gives
that $X$ is normal in $G$. It follows that $G$ is a Dedekind group in this
case.

Therefore it can be assumed that $G$ does not contain subgroups which are a
direct product of infinitely many non-trivial cyclic subgroups; hence all
abelian subgroups have finite total rank. Therefore the soluble radical of $%
G $ has a finite series in which each factor is abelian of finite total rank
(see \cite{C}); in particular $G$ has finite (Pr\"ufer) rank. Suppose that $G
$ is not a minimax group, so that it contains an abelian non-minimax
subgroup (see \cite{R72} Part 2, Theorem 10.35). Hence either $G^{\prime}$
is finite or $G/Z(G)$ is polycyclic-by-finite (see \cite{DM07}, Lemma 2.7).
Therefore $G^{\prime}$ is polycyclic-by-finite (see \cite{R72} Part 1,
p.115).
Let $H$ be any non-minimax subgroup of $G$. Then the abelian factor $H/H'$ is not minimax and so it has an homomorphic image which is the direct product of infinitely many non-trivial subgroups; thus $H$ is normal in $G$ by Lemma \ref{2}. Therefore all non-minimax subgroups of $G$ are normal and hence $G$ certainly has the weak minimal condition on non-normal subgroups. %
%
%
%
%
%
%
\end{proof}
%
%
%


\section{Real chain condition on non-modular subgroups}

A subgroup $H$ of a
group $G$ is said to be \textit{modular} if it is a modular
element of the lattice 
of all subgroups of $G$,
i.e., if $\left\langle H,X\right\rangle \cap Y=\left\langle
X,H\cap Y\right\rangle $ for all subgroups $X,Y$ of $G$ such that
$X\leq Y$ and $\left\langle H,X\right\rangle \cap Y=\left\langle
H,X\cap Y\right\rangle $ if $H\leq Y$. Lattices in which all
elements are modular are called \textit{modular}. Clearly every
normal subgroup is modular, but modular subgroups need not be
normal; moreover, a projectivity (i.e., an isomorphism between
subgroup lattices) maps any normal subgroup onto a modular
subgroup, thus modularity may be considered as a lattice
generalization of normality. A subgroup $H$ of a group $G$ is said
to be \textit{permutable} (or \textit{quasinormal}) if $HK=KH$ for
every subgroup $K$ of $G$; and the group $G$ is called
\textit{quasihamiltonian} if all its subgroups are permutable. It
is well-known that a subgroup is permutable if and only if it is
modular and ascendant, and that\ any modular subgroup of a locally
nilpotent group is always permutable (see \cite{S94}, Theorem
6.2.10). Groups with modular subgroup lattice, as well as quasihamiltonian groups, have been completely described and we refer to \cite{S94} as a general reference on (modular) subgroup lattice. In particular, recall that every non-periodic group with modular subgroup lattice is
quasihamiltonian, and that a periodic group is quasihamiltonian if
and only if it is a locally nilpotent group in which every
subgroup is modular. Moreover, any group with modular subgroup lattice is metabelian provided it is non-periodic or locally finite.

Recently, in \cite{DM18}, weak chain conditions on non-modular subgroups have been studied. It tourns out that for a generalised radical group, weak minimal and weak maximal condition on non-modular subgroups are both equivalent to the property that all non-minimax subgroups are modular and characterizes groups which either are soluble-by-finite and minimax or have modular subgroup lattice. Here we complete the description by considering RCC.

\medskip
\begin{lemma}
\label{1bis}Let $G$ be a group with RCC on non-modular subgroups having section $H/K$ which is a direct product of infinitely many non-trivial subgroups. If  $x$ is an element of $G$ such that $\left\langle x\right\rangle
\cap H\leq K$,  then there exists a subgroup $L$
of $H$ such that both $L$ and $\left\langle x,L\right\rangle $ are modular
subgroup of $G$.
\end{lemma}

\begin{proof}
It can be assumed that $H/K=\underset{i\in \mathbb{Q}}{\text{Dr}}H_{i}/K$ with each $H_i\ne K$. For $r\in \mathbb{R}$, let $L_{r}/K=\underset{i<r}{\text{Dr}}H_{i}/K$, so that each $L_{r}$ is a modular subgroup of $G$ by Lemma \ref{2}. Let $r_1,r_2\in\mathbb{R}$ such that $r_1<r_2$, then  
$\left\langle x,L_{r_1}\right\rangle\leq\left\langle
x,L_{r_2}\right\rangle$ and we claim that $\left\langle x,L_{r_1}\right\rangle\neq\left\langle
x,L_{r_2}\right\rangle$. In fact, if were $\left\langle x,L_{r_1}\right\rangle =\left\langle
x,L_{r_2}\right\rangle $, since $\left\langle x\right\rangle \cap L_{r_1}\leq
\left\langle x\right\rangle \cap H\leq K\leq L_{r_2}$ and $L_{r_1}\leq L_{r_2}$ are modular subgroups, 
it would be $L_{r_1}=\left\langle x,L_{r_2}\right\rangle
\cap L_{r_1}=\left\langle L_{r_2},\left\langle x\right\rangle \cap
L_{r_1}\right\rangle =L_{r_2}$, a contradiction. Therefore $(\left\langle
x,L_{r}\right\rangle )_{r\in \mathbb{R}}$ is an $\mathbb{R}$-chain, and
hence there exists some $\left\langle x,L_{r}\right\rangle $ which is
modular in $G$.
\end{proof}

\medskip
In order to reduce the study of generalised radical groups to radical-by-finite groups, we need to prove that locally finite groups with RCC on non-modular subgroups are soluble-by-finite (see Proposition \ref{GenRad}), this will be follow from next result. Recall that the class of groups with modular subgroup lattice is local, i.e. a group $G$ has modular subgroup lattice if and only if each finitely generated subgroup of $G$ has modular subgroup lattice (see \cite{Pr}, Corollario\;1.4).

\medskip 
\begin{lemma}
\label{LFM}Let $G$ be a locally finite group with RCC on non-modular subgroups. Then $G$ either is a Chernikov group or has modular subgroup lattice.
\end{lemma}

\begin{proof}
If $G$ is not a Chernikov group, it contains an abelian subgroups whose
socle $A$ is infinite (see \cite{Shunkov}). Let $X$ be any finite subgroup
of $G$. Replacing $A$ by a suitable subgroup, it can be assumed that $A\cap
X=\{1\}$. If $x$ is any element of $X$, then Lemma \ref{1bis} yields that
there exists a subgroup $L$ of $A$ such that both $L$ and $\left\langle
L,x\right\rangle $ are modular in $G$; thus $\left\langle L,x\right\rangle
\cap X=\left\langle x,L\cap X\right\rangle =\left\langle x\right\rangle $ is
modular in $X$. Therefore all cyclic subgroups of $X$ are modular, and so $X$
has modular subgroup lattice. Since the class of groups with modular
subgroup lattice is local, we have that $G$ itself has modular subgroup lattice.
\end{proof}

\medskip 
\begin{lemma}
\label{5}Let $G$ be a radical-by-finite group with RCC on non-modular subgroups. Then $G$
is either a soluble-by-finite minimax group or has modular subgroup lattice.
\end{lemma}

\begin{proof}
Let $H$ be any non-minimax subgroup of $G$. By the above quoted result of \cite{DM18}, it is enought to prove that if $H$ is a modular subgroup. 

There exists an abelian subgroup $A$ of $H$ which is not minimax (see \cite{R72} Part 2, Theorem 10.35), and so $A$ has an homomorphic image $A/B$ which is the
direct product of infinitely many non-trivial subgroups. Let $x$ be any
element of $H\smallsetminus B$; clearly replacing $A/B$ with a suitable
direct factor which is likewise the direct product of infinitely many
non-trivial subgroups, it can be assumed that $\left\langle x\right\rangle
\cap A\leq B$. Therefore Lemma \ref{1bis} yields that there exists a
subgroup $L_{x}$ of $A$ such that $L_{x}$ and $\left\langle
x,L_{x}\right\rangle $ are modular subgroup of $G$. 
Since the join of modular subgroups is likewise modular (see for instance \cite{Pr}, Proposizione\;1.2), it follows that $%
H=\left\langle \left\langle x,L_{x}\right\rangle :x\in H\smallsetminus
B\right\rangle $ is modular in $G$ and the proof is completed.
\end{proof}

\medskip
We are now in position to prove the Framework Statement when $\chi=m$.
\medskip

\begin{theorem}
\label{teom} Let $G$ be a generalised radical group. Then the following are
equivalent:

\begin{itemize}
\item[(i)] $G$ satisfies the weak minimal condition on non-modular
subgroups;

\item[(ii)] $G$ satisfies the weak maximal condition on non-modular
subgroups;

\item[(iii)] $G$ satisfies the weak double condition on non-modular
subgroups.

\item[(iv)]  $G$ satisfies the real chain condition on non-modular subgroups.

\item[(v)] either $G$ is a soluble-by-finite minimax group or all subgroups of $G$ are modular.
\end{itemize}
\end{theorem}

\begin{proof}
As already noted, conditions $(i)$, $(ii)$ and $(v)$ are equivalent, and clearly imply $(iii)$; moreover, $(iii)$ implies $(iv)$ by Proposition \ref{WDDC_implies_deviation}. On the other hand, since any locally finite group with modular subgroup lattice is soluble (see \cite{S94}, Theorem 2.4.21), if $G$ satisfies $(iv)$, then $G$ is radical-by-finite by Lemma \ref{LFM} and Proposition \ref{GenRad} and so it satisfies $(v)$ by Lemma \ref{5}. The theorem is proved.
\end{proof}
%
%
%
%
%


\medskip
In \cite{DM18}, weak chain conditions on non-permutable subgroups have been also considered and it was proved that all results on weak chain conditions on non-modular subgroups have a corresponding with non-permutable subgroups. Here the wished results for groups in which the poset of all non-permutable subgroups has RCC can be obtained just replacing modular subgroups with permutable subgroups in the above arguments or, in an independent way, as a consequence of the following.

\begin{lemma}\label{xqn} Let $G$ be a periodic locally soluble group with RCC on non-permutable subgroups. Then either $G$ is a Chernikov group or all subgroups of $G$ are permutable.
\end{lemma}

\begin{proof}
Assume that $G$ is not a Chernikov group and let $x,y\in G$. Clearly $\langle x,y\rangle$ is finite and so, since $G$ is locally soluble, there exists an abelian $\langle x,y\rangle$-invariant subgroup $A$ which does not satisfy the minimal condition (see \cite{Z74}). Replacing $A$ by its socle, it can be assumed that $A$ is the direct product of infinitely many cyclic groups of prime order. Application of Lemma \ref{anor} gives that $A$ contains a subgroup $B$ which is the direct product of infinitely many non-trivial finite $\langle x,y\rangle$-invariant subgroups. Clearly it can be assumed that $B\cap \langle x,y\rangle =\{1\}$, hence Lemma \ref{1} yields that $B$ contains a normal subgroup $B^{\ast}$ such that $\langle x\rangle B^{\ast}=B^{\ast}\langle x\rangle$ is permutable in $B$. Hence
\[
\langle x\rangle \langle y\rangle \subseteq (B^{\ast}\langle x\rangle)  \langle y\rangle =\langle y\rangle(B^{\ast}\langle x\rangle)=\langle y\rangle \langle x\rangle B^{\ast}
\]
and so, since $B\cap \langle x,y\rangle =\{1\}$, it follows that $\langle x\rangle \langle y\rangle \subseteq \langle y\rangle \langle x\rangle$. Similarly  $\langle y\rangle \langle x\rangle \subseteq \langle x\rangle \langle y\rangle$ and hence  $\langle x\rangle \langle y\rangle = \langle y\rangle \langle x\rangle$.
Therefore all (cyclic) subgroups of $G$ are permutable.
\end{proof}

\begin{theorem}
\label{teoqn} Let $G$ be a generalised radical group. Then the following are
equivalent:

\begin{itemize}

\item[(i)] $G$ satisfies the weak minimal condition on non-permutable
subgroups;

\item[(ii)] $G$ satisfies the weak maximal condition on non-permutable
subgroups;

\item[(iii)] $G$ satisfies the weak double condition on non-permutable
subgroups.

\item[(iv)]  $G$ satisfies the real chain condition on non-permutable subgroups.

\item[(v)] either $G$ is a soluble-by-finite minimax group or all subgroups of $G$ are permutable.
\end{itemize}
\end{theorem}

\begin{proof} As in Theorem \ref{teom}, it is enough to prove that $(iv)$ implies $(v)$. Hence assume that $G$ satisfies $(iv)$. Theorem \ref{teom} yields that either $G$ is a soluble-by-finite minimax groups or has modular subgroup lattice, so that application of Theorem 2.4.11 of \cite{S94} and Lemma\;\ref{xqn} give that $(v)$ holds, and so the theorem is proved.\end{proof}

%

\medskip

\section{Real chain condition on non-subnormal subgroups}

The weak minimal and the weak maximal condition on non-subnormal subgroups have been considered in \cite{KSwmin} and in \cite{KSwmax} respectively. It turns out that {\it if $G$ is a generalised radical group $G$ satisfying the weak minimal condtions on non-subnormal subgroups, then either $G$ is a soluble-by-finite minimax group or any subgroup of $G$ is subnormal}. On the other hand, there exists non-minimax groups satisfying the weak maximal condition on non-subnormal subgroups which still have non-subnormal subgroups. Indeed, if $G=A\rtimes\langle g\rangle$ where $A=\underset{i\in\mathbb{N}}{%
\mathrm{Dr}}{\langle a_{i}\rangle}$ is an infinite elementary abelian $p$-group ($p$ prime) and $g$ is the automorphism of infinite order  of $A$ such that $[a_1, g]=1$ and $[a_{i+1},g]=a_i$ for all $i\geq 1$, then $G$ is an hypercentral non-minimax group satisfying the weak maximal condition on non-subnormal subgroups which is not a Baer group (see \cite{KSwmax}). Recall here that the {\it Baer radical} of a group is the subgroups generated by all cyclic subnormal subgroups and a group is said to be a {\it Baer group} if it concides with its Baer radical; in particular, in a Baer group all finitely generated subgroups are subnormal and nilpotent.

\medskip
Notice that the above example $G=A\rtimes \langle g\rangle$ does not satisfy the weak minimal condition on non-subnormal subgroups but the poset of all non-subnormal subgroups of $G$ has deviation (see the introduction of \cite{KS}), and so also RCC by Proposition \ref{WDDC_implies_deviation}. Hence Framework Statement cannot be proved in his form when $\chi=sn$ is the property for a subgroup to be subnormal . However, for locally finite groups the weak minimal condition on non-subnormal subgroups is equivalent to the weak maximal condtion on non-subnormal subgroups, and here we are able to prove the Framework Statement when $\chi=sn$  within the universe of periodic soluble groups, improving Theorem 1 of \cite{GKR1} which concernes with soluble periodic groups with deviation on the poset of non-subnormal subgroups.

\begin{lemma}
\label{snper}Let $G$ be a periodic group with RCC on non-subnormal subgroups. If $G$
contains an abelian subgroup $A$ which does not statisfy the minimal condition,
then $G$ is a Baer group.
\end{lemma}

\begin{proof} Replacing $A$  by its socle it can be supposed that $A$ is
the direct product of infinitely many cyclic non-trivial subgroups. As a consequece of Lemma \ref{catena-} it can be obtained that all cyclic subgroups of $A$ are
subnormal in $G$, hence $A$ is contained in the Baer radical $R$ of $G$ and
hence $R$ does not satisfy the minimal condition. Let $g$ be any element of $%
G$. Then $\langle R,g\rangle$ is locally soluble and hence there is no loss
of generality if we assume that $A$ is $\langle g\rangle$-invariant (see 
\cite{Z74}). Then $A$ has finite index in $\langle A,g\rangle$ and hence all
subgroups of $A$ are almost normal in $G$. Thus Lemma \ref{anor} yelds that $%
A$ contains a subgroup which is the direct produt of infinitely many
finitely generated $\langle g\rangle$-invariant non-trivial subgroups, and
so it follows from Lemma \ref{catena-} that $g\in R$. Thus $G=R$ is a Baer
group.
\end{proof}

\medskip

\begin{corollary}
\label{corsn} Let $G$ be a locally finite group RCC on non-subnormal subgroups. Then $G$ is either a Chernikov group or a Baer group.
\end{corollary}

\begin{proof}
This follows from \cite{Shunkov} and Lemma \ref{snper}.
\end{proof}

\medskip In our argument we need the following easy remark. \medskip

\begin{lemma}
\label{a1} Let $G$ be a group and let $N$ be a normal subgroup. If $N$
satisfies maximal (resp. minimal) condition on $G$-invariant subgroups and $%
G/N$ satisfies the weak maximal (resp. weak minimal) condition on normal
subgroups, then $G$ satisfies the weak maximal (resp. weak minimal)
condition on normal subgrops.
\end{lemma}

\begin{proof}
Let $(G_i)_{i\in\mathbb{N}}$ be an ascending chain of normal subgroups of $G$. Then $(G_i\cap N)_{i\in\mathbb{N}}$ is an ascending chain of $G$-invariant
subgroups of $N$ and hence there exists a positive integer $n$ such that the
index $G_{i+1}\cap N=G_i\cap N$ for any $i\geq n$. On the other hand, $%
(G_iN/N)_{i\geq n}$ is an ascending normal chain and so there exits a
positive integer $m\geq n$ such that the index $|G_{i+1}N/N:G_iN/N|$ is
finite for any $i\geq m$. Then for every $i\geq m$ we have that the index 
\begin{align*}
|G_{i+1}N:G_iN|=&|G_{i+1}(G_iN):G_iN|=|G_{i+1}:G_iN\cap G_{i+1}|= \\
=&|G_{i+1}:G_i(N\cap G_{i+1})|=|G_{i+1}:G_i(N\cap G_{i})|=|G_{i+1}:G_{i}|
\end{align*}
is finite. Thus the result with weak maximal conditions is proved, the
corresponding resut with weak minimal conditions can be proved similarly.
\end{proof}

\medskip Recall that if $G$ is a periodic Baer group, any subnormal abelian
divisible subgroup is contained in the centre of $G$ (see for instance \cite{KSwmin}, Lemma 5.1).

\begin{lemma}
\label{pru} Let $G=AB$ a periodic Baer  group with RCC on non-subnormal subgroups, where $A$ is an abelian normal subgroup of $G$ and $B$ is abelian divisible.  Then $G$ is abelian.
\end{lemma}

\begin{proof} 
Let $G$ be a counterexample; in particular, the result just quoted above gives that $B$ is not subnormal in $G$ and so $G$ is not nilpotent. Since $B$ is a direct product of Pr\"ufer subgroups, some of them $P$ does not centralize $A$. Hence also $AP$ is couterexample and so it can be assumed that $B=P$ is a Pr\" ufer group. Since $A\cap B\leq C_B(A)\leq B_G$ and $G/B_G$ is still a counterexample, it can be assumed also that $A\cap B=C_B(A)=B_G=\{1\}$. 

Assume that there exists an $\mathbb{R}$-chain $(X_i)_{i\in\mathbb{R}}$ of $
B$-invariant (proper) subgroups of $G'$; clearly each $X_i$ is a normal subgroup of $G$. Since $A\cap B=\{1\}$ and $G'\leq A$, we obtain
that $(X_iB)_{i\in\mathbb{R}}$ is an $\mathbb{R}$-chain of subgroups of $G$
and hence, since $G$ has RCC on non-subnormal subgroups, there exists an $r\in\mathbb{R}$ such that $X_rB$ is a subnormal subgroup of $G$. Therefore 
$X_rB/X_r\simeq B$ is a subnormal Pr\"ufer subgroup of the periodic Baer group $G/X_r$, hence $X_rB/X_r\leq Z(G/X_r)$. It follows that $G'=[A,B]\leq [X_rB,G]\leq X_r$ and hence $X_r=G'$, a contradiction. Therefore the (complete) poset of all $B$%
-invariant subgroups of $G'$ does not contain $\mathbb{R}$-chains and so it has deviation by
Propositon \ref{compl}. Since $G'$ can be considered as a $\mathbb{Z}B$%
-module, as $B$ acts on $G'$ by conjugation, it follows that $G'$ contains a
finite series of $B$-invariant subgroups 
\begin{equation*}
\{1\}=Y_0\leq Y_1\leq \cdots\leq Y_k=G'
\end{equation*}
whose factors $Y_{i}/Y_{i-1}$ satisfy either the minimal or the maximal
condition on $B$-invariant subgroups (see \cite{T03}, Theorem 4.3). 
Let $i\leq k$. Since any Pr\"ufer group satisfy both the weak minimal and
the weak maximal condition, it follows from Lemma \ref{a1} that the factor
group $Y_{i}B/Y_{i-1}$ satisfy either the weak minimal or the weak maximal
condition on normal subgroups. Hence the Baer group $Y_{i}B/Y_{i-1}$ is
nilpotent (see \cite{KW}) and so even abelian since $Y_{i-1}B/Y_{i-1}\simeq
B$. It follows that $Y_{i-1}B$ is normal in $Y_{i}B$ for every $i\leq k$. Thus $B=BY_0$ is subnormal in $BY_k=BG'$ and hence also in $G$, a contradiction which concludes the proof.
\end{proof}
\medskip
We are now in a position to prove the main result of this section.
\medskip
%
\begin{theorem}
\label{teosn} Let $G$ be a periodic soluble group. Then the following are
equivalent:

\begin{itemize}

\item[(i)] $G$ satisfies the weak minimal condition on non-subnormal
subgroups;

\item[(ii)] $G$ satisfies the weak maximal condition on non-subnormal
subgroups;

\item[(iii)] $G$ satisfies the weak double condition on non-subnormal
subgroups.

\item[(iv)]  $G$ satisfies the real chain condition on non-subnormal subgroups.

\item[(v)] either $G$ is a Chernikov group or all subgroups of $G$ are subnormal.
\end{itemize}
\end{theorem}

\begin{proof}
Clearly, $(v)$ implies both $(i)$ and $(ii)$, and and each of them imply $(iii)$; moreover, $(iii)$ imply $(iv)$ by Proposition \ref{WDDC_implies_deviation}. Thus we have to prove that $(iv)$ imply $(v)$. Let $G$ be a non-Chernikov periodic soluble group satisfying $(iv)$, so that $G$ is a Baer group by Corollary\; \ref{corsn}. 

Assume for a contradiction that the statement is false, and among all
counterexamples for which $G$ has smallest derived length choose one such
that $G$ contains a non-subnormal subgroup $X$ whose derived length is
minimal possible. If $A$ is the smallest non-trivial term of the derived
series of $G$, the minimality of the derived length of $G$ gives that $XA$
is subnormal in $G$, so that $X$ cannot be subnormal in $XA$. Moreover, the
minimality of the derived length of $X$ yields that $X^{\prime}$ is
subnormal in $G$ of defect $k$, say. The intersection $X \cap A$ is a normal
subgroup of $XA$ and the factor $XA/(X \cap A)$ is again a minimal
counterexample, so it can be assumed that $G = XA$ and $X \cap A =\{1\}$. 
%
Put $A_0 = A$ and $A_i = [A,_iX']$ for each positive integer $i\leq k$.
Clearly every $A_i$ is a normal subgroup of $G$ and $A_k =\{1\}$. The
consideration of the chain 
\begin{equation*}
G = A_0X \geq A_1X\geq \cdots\geq A_kX= X
\end{equation*}
gives that there exists a positive integer $j\leq k$ such that $A_jX$ is not
subnormal in $A_{j-1}X$. Since $A_jX'$ is contained in $A_jX$ and it is
normal in $A_{j-1}X$, the factor group $A_{j-1}X/A_jX'$ is likewise a minimal
counterexample. Thus we may replace $G$ by $A_{j-1}X/A_jX'$ and $X$ by $%
A_{j}X/A_jX'$, i.e. it can be supposed that $X$ is abelian.

Since any abelian group which is not minimax has an homomorphic image which
is the direct product of infinitely many non-trivial subgroups and since $X$
is not subnormal in $G$, Lemma \ref{2} gives that $X$ is a Chernikov group.
Let $D$ be the largest divisible subgroup of $X$ and let $F$ be a finite
subgroup, such that $X = DF$. Then $[A,D]=\{1\}$ by Lemma \ref{pru}, so that 
$AD$ is a normal abelian subgroup of $G=AX=(AD)F$ and hence $G$ is nilpotent
because $F$ is finite and $G$ is a Baer group. This contradiction completes
the proof. %
\end{proof}



\end{document}